\documentclass{amsart}

\usepackage{amssymb}
\usepackage{amsmath}
\usepackage{amsthm}

\newtheorem {theorem}{Theorem}%[section]
\newtheorem {proposition}{Proposition}
\newtheorem {corollary}{Corollary}
\newtheorem {lemma}{Lemma}

\newtheorem{prb}{Problem}
\newtheorem{rem}{Remark}

\numberwithin{equation}{section}

\newcommand{\sln}{\mathrm{SL}_n}
\newcommand{\gln}{\mathrm{GL}_n}
\newcommand{\glz}{\mathrm{GL}_2}
\newcommand{\sla}{\mathrm{SL}}
\newcommand{\slz}{\mathrm{SL}_2}
\newcommand{\ut}{\mathrm{ut}}
\newcommand{\EE}{E}
\newcommand{\bsr}{\mathrm{bsr}}
\newcommand{\Tbb}{\mathbb{T}}
\newcommand{\OO}{\mathcal O}
\newcommand{\Dbb}{\mathbb D}
\newcommand{\Rbb}{\mathbb R}
\newcommand{\Cbb}{\mathbb C}
\newcommand{\Nbb}{\mathbb N}

\newcommand{\za}{\zeta}

\begin{document}

\title[Factorization by elementary matrices]
{Factorization by elementary matrices, null-homotopy
and products of exponentials for invertible matrices over rings}

\author{Evgueni Doubtsov}

\address{St.~Petersburg Department
of V.A.~Steklov Mathematical Institute,
Fontanka 27, St.~Petersburg 191023, Russia}

\email{dubtsov@pdmi.ras.ru}

\author{Frank Kutzschebauch}

\address{Institute of Mathematics,
              University of Bern,
              Sidlerstrasse 5, CH-3012 Bern,
              Swit\-zerland}

\email{frank.kutzschebauch@math.unibe.ch}

\thanks{Frank Kutzschebauch
was supported by Schweizerische Nationalfonds Grant 200021-178730.}


\date{}

\begin{abstract}
Let $R$ be a commutative unital ring.
A well-known factorization problem is whether any matrix in $\mathrm{SL}_n(R)$
is a product of elementary matrices with entries in $R$.
To solve the problem, we use two approaches based on the notion
of the Bass stable rank and on construction of a null-homotopy.
Special attention is given to the case, where $R$
is a ring or Banach algebra of holomorphic functions.
Also, we consider a related problem on representation
of a matrix in $\mathrm{GL}_n(R)$
as a product of exponentials.
\end{abstract}

%%% ----------------------------------------------------------------------
\maketitle
%%% ----------------------------------------------------------------------

\section{Introduction}\label{s_int}

Let $R$ be an associative, commutative, unital ring.
A well-known factorization problem is whether any matrix in $\sln(R)$
is a product of elementary (equivalently, unipotent)
matrices with entries in $R$.
Here the elementary matrices are those which have units on the diagonal 
and zeros outside the diagonal, except one non-zero entry.
In particular, for $n=3,4,\dots$, Suslin \cite{Su77} proved that the problem is solvable
for the polynomials rings $\Cbb[\Cbb^m]$, $m\ge 1$.
For $n=2$, the required factorization for $R= \Cbb[\Cbb^m]$ does not always exist;
the first counterexample was constructed by Cohn \cite{Co66}.

In the present paper, we primarily consider the case, where
$R$ is a functional Banach algebra.
So, let $\OO(\Dbb)$ denote the space of holomorphic functions
on the unit disk $\Dbb$ of $\Cbb$.
Recall that the disk-algebra
$A(\Dbb)$ consists of $f\in\OO(\Dbb)$
extendable up to continuous functions on the closed disk $\overline{\Dbb}$.
The disk-algebra $A(\Dbb)$
and the space $H^\infty(\Dbb)$ of bounded holomorphic functions on $\Dbb$
 may serve as good working examples
for the algebras under consideration.
%\subsection*{}

In fact, we propose two approaches to the factorization problem.
The first one is based on construction of a null-homotopy; see Section~\ref{s_null}.
This method applies to the disk-algebra and similar algebras.
The second approach is applicable to rings whose Bass stable rank is equal to one; see Section~\ref{s_suff}.
This methods applies, in particular, to $H^\infty(\Dbb)$.

Also, the factorization problem is closely related to the following natural question:
whether a matrix $F\in \gln(R)$ is representable as a product of exponentials,
that is, $F = \exp G_1\dots \exp G_k$ with $G_j \in M_n(R)$.
For $n=2$ and matrices with entries in a Banach algebra,
this question was recently considered in \cite{MR18}.
In Section~\ref{s_exp}, we obtain results related to this question
with emphasis on the case, where
$R=\OO(\Omega)$ and $\Omega$
is an open Riemann surface.

\section{Factorization and null-homotopy}\label{s_null}

Given $n\ge 2$ and an associative, commutative, unital ring $R$,
let $\EE_n(R)$ denote the set of those $n\times n$ matrices
which are representable as products of elementary
matrices with entries in $R$.
%The minimal number of the corresponding multiplies is denoted by $t_n(R)$.

For a unital commutative Banach algebra $R$,
an element $X\in \sln(R)$ is said to be null-homotopic if $X$
is homotopic to the unity matrix, that is, there exists
a homotopy $X_t : [0,1] \to \sln(R)$ such that $X_1=X$ and $X_0$ is the unity matrix.

We will use the following theorem:

\begin{theorem}[{\cite[\S 7]{Mil71}}]\label{t_homot}
Let $A$ be a unital commutative Banach algebra and let $X\in \sln(A)$.
The following properties are equivalent:
\begin{itemize}
  \item[(i)] $X\in\EE_n(A)$;
  \item[(ii)] $X$ is null-homotopic.
\end{itemize}
\end{theorem}

%Let $\hol(\Dbb)$ denote the space of holomorphic functions in the unit disk $\Dbb$.
To give an illustration of Theorem~\ref{t_homot}, consider
the disk-algebra $A(\Dbb)$.

\begin{corollary}\label{c_AD}
For $n=2,3,\dots$, $\EE_n(A(\Dbb)) = \sln(A(\Dbb))$.
\end{corollary}
\begin{proof}
We have to show that $\EE_n(A(\Dbb)) \supset \sln(A(\Dbb))$.
So, assume that
\[
F = F(z) = \begin{pmatrix}
  f_{11}(z) & &f_{1n}(z) \\
       &\ddots & \\
  f_{n1}(z) & &f_{nn}(z)
\end{pmatrix}
\in \sln(A(\Dbb)).
\]

Define
\begin{equation}\label{e_Xt}
F_t(z) = F(tz)\in \sln(A(\Dbb)), \quad 0\le t \le 1, \ z\in\Dbb.
\end{equation}
Given an $f\in A(\Dbb)$, let $f_t(z) = f(tz)$, $0\le t \le 1$, $z\in \Dbb$.
Observe that $\|f_t - f\|_{A(\Dbb)} \to 0$ as $t\to 1-$.
Applying this observation to the entries of $F_t$, we conclude that $F$
is homotopic to the constant matrix $F(0)$.
Since $\sln(\Cbb)$ is path-connected, the constant matrix $F(0)$
is homotopic to the unity matrix.
So, it remains to apply Theorem~\ref{t_homot}.
\end{proof}

\section{Factorization and Bass stable rank}\label{s_suff}

\subsection{Definitions}
Let $R$ be a commutative unital ring. An element
$(x_1, \dots, x_k) \in R^k$ is called \textsl{unimodular} if
\[
\sum_{j=1}^k x_j R =R.
\]
Let $U_k(R)$ the set of all unimodular elements in $R^k$.

An element $x=(x_1, \dots, x_{k+1}) \in U_{k+1}(R)$ is called \textsl{reducible}
if there exists $(y_1, \dots, y_k)\in R^k$ such that
\[
(x_1 + y_1 x_{k+1}, \dots, x_k + y_k x_{k+1}) \in U_k (R).
\]
The \textsl{Bass stable rank} of $R$, denoted by $\bsr(R)$ and introduced in \cite{Ba64}, is the least $k\in\Nbb$ such that every $x\in U_{k+1} (R)$ is reducible. If there is no such $k\in\Nbb$,
then we set $\bsr(R) =\infty$.

\begin{rem}\label{r_bsr1}
The identity $\bsr(R) =1$ is equivalent to the following property:
For any $x_1, x_2\in R$ such that $x_1R + x_2R =R$,
there exists $y\in R$ such that $x_1 + y x_2 \in R^*$.
\end{rem}

\subsection{A sufficient condition for factorization}

\begin{theorem}\label{t_sbr1_2}
Let $R$ be a unital commutative ring and $n\ge 2$.
If $\bsr (R) =1$, then $\EE_n(R) = \sln(R)$.
\end{theorem}
\begin{proof}
First, assume that $n=2$. %This case may serve as a model for a general argument.
Let
\[
X = \begin{pmatrix}
  x_{11} &x_{12} \\
  x_{21} &x_{22}
\end{pmatrix}
\in \slz(R).
\]
Since $\det X =1$, we have
\[
x_{21}R + x_{11}R = R.
\]
Hence, using the assumption $\bsr(X)=1$ and Remark~\ref{r_bsr1}, we conclude that
there exists $y\in R$ such that
\begin{equation}\label{e_bsr1}
\alpha = x_{21} + y x_{11} \in R^{*}.
\end{equation}

Now, we have
\[
\begin{pmatrix}
  1      &0 \\
  y      &1
\end{pmatrix}
X =
\begin{pmatrix}
  x_{11} &x_{12} \\
  \alpha &*
\end{pmatrix}.
\]
Next, using \eqref{e_bsr1}
we obtain
\[
\begin{pmatrix}
  1      &(1-x_{11})\alpha^{-1} \\
  0      &1
\end{pmatrix}
\begin{pmatrix}
  x_{11} &x_{12} \\
  \alpha &*
\end{pmatrix}
=
\begin{pmatrix}
  1      &* \\
  \alpha &*
\end{pmatrix}.
\]
Finally, we have
\[
\begin{pmatrix}
  1       &0 \\
  -\alpha &1
\end{pmatrix}
\begin{pmatrix}
  1      &* \\
  \alpha &*
\end{pmatrix}
=
\begin{pmatrix}
  1      &* \\
  0      &x_0
\end{pmatrix}.
\]
Since the determinant of the last matrix is equal to one, we conclude that $x_0=1$.
Therefore, the $X$ is representable as a product of four multipliers.

For $n\ge 3$, let
\[
X = \begin{pmatrix}
  x_{11} & \\
  \vdots &* \\
  x_{n1} &
\end{pmatrix}
\in \sln(R).
\]
Since $\det X =1$, there exist $\alpha_1, \dots, \alpha_n\in R$ such that
$\alpha_1 x_{11}+\dots + \alpha_{n-1}x_{n-1 1} + \alpha_n x_{n 1} =1$.
Therefore,
\[
x_{n 1}R + \left(\sum_{i=1}^{n-1}\alpha_{i}x_{i 1}  \right)R = R.
\]
Applying the property $\bsr R =1$, we obtain $y\in R$ such that
\[
x_{n 1} + y \left(\sum_{i=1}^{n-1}\alpha_{i}x_{i 1}  \right) := \alpha \in R^*.
\]
Put
\[
L=
\begin{pmatrix}
   1       &           & &\\
           &1          &\mathbf{0} & \\
           &\mathbf{0} &\ddots &\\
\alpha_1 y &\dots &\alpha_{n-1}y &1
\end{pmatrix}.
\]
Then
\[
LX =
\begin{pmatrix}
  x_{11} & \\
  \vdots &* \\
  x_{n-1 1} & \\
  \alpha &
\end{pmatrix}.
\]
Multiplying by the upper triangular matrix
\[
U_1
=\begin{pmatrix}
   1       &          & &           &(1-x_{11})\alpha^{-1}\\
           &1         & &\mathbf{0} &-x_{21}\alpha^{-1} \\
         &\mathbf{0} &\ddots & &\dots \\
         & & &1 &-x_{n-1 1}\alpha^{-1}\\
         & & & &1
\end{pmatrix},
\]
we obtain
\[
U_1 LX =
\begin{pmatrix}
  1 & \\
  0      & \\
  \vdots &* \\
  0      & \\
  \alpha &
\end{pmatrix}.
\]
Now, put
\[
\widetilde{L}
=
\begin{pmatrix}
   1       &           & &\\
           &1          &\mathbf{0} & \\
   0      &\mathbf{0} &\ddots &\\
-\alpha    &0 & &1
\end{pmatrix}.
\]
We have
\[
\widetilde{L} U_1 LX =
\begin{pmatrix}
   1       & *         &* &* \\
    0       &          & & \\
    \vdots  & &Y_1 &\\
   0        & & &
\end{pmatrix}.
\]
Observe that $Y_1\in \mathrm{SL}_{n-1}(R)$.
So, arguing by induction, we obtain
\[
\left(\prod_{i=1}^{n-1} \widetilde{L}_i U_i L_i\right)X =
\begin{pmatrix}
   1       &       &*\\
           &\ddots &\\
    \mathbf{0}       &           &1
\end{pmatrix}
:=U
\]
or, equivalently,
\[
\left(\prod_{i=1}^{n-1} \mathcal{L}_i U_i \right) L_{n-1} X =U,
\]
where $\mathcal{L}_i$ are lower triangular matrices.
So, we conclude that every $X\in\sln(R)$ is a product of $2n$
unipotent upper and lower triangular matrices.
\end{proof}

\begin{corollary}\label{c_sbr1_2}
Let $A$ be a unital commutative Banach algebra such that $\bsr (A) =1$.
If $X\in \sln(A)$, then $X$ is null-homotopic.
\end{corollary}
\begin{proof}
It suffices to combine Theorems~\ref{t_homot} and \ref{t_sbr1_2}.
\end{proof}

\newpage

\subsection{Examples of algebras $A$ with $\bsr(A)=1$}

\subsubsection{Disk-algebra $A(\Dbb)$}
By Corollary~\ref{c_AD}, $\EE_n(A(\Dbb)) = \sln(A(\Dbb))$.
%For $n=2$,
Theorem~\ref{t_sbr1_2} provides a different proof of this property.
Indeed, Jones, Marshall and Wolff \cite{JMW86} and Corach and Su\'{a}rez \cite{CS85}
proved that $\bsr(A(\Dbb))=1$, so Theorem~\ref{t_sbr1_2} applies.

\subsubsection{Algebra $H^\infty(\Dbb)$}
Let $f\in H^\infty(\Dbb)$. If $\|f_r -f\|_\infty \to 0$ as $r\to 1-$,
then clearly $f\in A(\Dbb)$. So the homotopy argument used for $A(\Dbb)$
is not applicable to $H^\infty(\Dbb)$. However, Treil \cite{Tr92}
proved that $\bsr(H^\infty(\Dbb))=1$, hence, Theorem~\ref{t_sbr1_2}
holds for $R=H^\infty(\Dbb)$. Also,
Corollary~\ref{c_sbr1_2} guarantees that any
$F\in \sln(H^\infty(\Dbb))$ is null-homotopic.

\subsubsection{Generalizations of $H^\infty(\Dbb)$}
Tolokonnikov \cite{To95} proved that $\bsr(H^\infty(G))=1$
for any finitely connected open Riemann surface $G$ and for certain
infinitely connected planar domains $G$ (Behrens domains).
In particular, any $F\in \sln(H^\infty(G))$ is null-homotopic.
However, even in the case $G=\Dbb$ the homotopy in question is not explicit.
So, probably it would be interesting to give a more explicit construction
of the required homotopy.

Let $\Tbb=\partial\Dbb$ denote the unit circle.
Given a function $f\in H^\infty(\Dbb)$, it is well-known that
the radial limit $\lim_{r\to 1-} f(r\za)$ exists for almost all $\za\in\Tbb$
with respect to Lebesgue measure on $\Tbb$. So, let
$H^\infty(\Tbb)$ denote the space of the corresponding radial values.
It is known that $H^\infty(\Tbb) + C(\Tbb)$ is an algebra,
moreover, $\bsr(H^\infty(\Tbb) + C(\Tbb))=1$; see \cite{MW10}.

Now, let $B$ denote a Blaschke product in $\Dbb$.
Then $\Cbb + B H^\infty(\Dbb)$ is an algebra.
It is proved in \cite{MSW10} that $\bsr(\Cbb + B H^\infty(\Dbb))=1$.

\subsection{Examples of algebras $A$ with $\bsr(A)>1$}

\subsubsection{Algebra $A_{\Rbb}(\Dbb)$}
Each element $f$ of the disk-algebra $A(\Dbb)$ has a unique representation
\begin{equation}\label{e_f_series}
f(z) = \sum_{j=0}^\infty a_j z^j, \quad z\in \Dbb.
\end{equation}
The space $A_{\Rbb}(\Dbb)$ consists of those $f\in A(\Dbb)$ for which $a_j\in\Rbb$ for all $j=0,1\dots$
in \eqref{e_f_series}.
As shown in \cite{MW09}, $\bsr(A_{\Rbb}(\Dbb))=2$.
Nevertheless, the following result holds.

\begin{proposition}\label{p_A_Rbb}
For $n=2, 3, \dots$, $\EE_n(A_{\Rbb}(\Dbb)) = \sln(A_{\Rbb}(\Dbb))$.
\end{proposition}
\begin{proof}
For a function $f\in A_{\Rbb}(\Dbb)$,
we have $f_t \in A_{\Rbb}(\Dbb)$ or all $0\le t < 1$.
Hence, given a matrix $F\in \sln(A_{\Rbb}(\Dbb))$, we have $F_t\in \sln(A_{\Rbb}(\Dbb))$,
where $F_t$ is defined by \eqref{e_Xt}.
Since $\|f_t - f\|_{A_{\Rbb}(\Dbb)} \to 0$ as $t\to 1-$, $F$ is homotopic to the constant matrix $F_0\in \sln(\Cbb)$.
Hence, $F$ is homotopic to the unity matrix.
Therefore, $F\in \EE_n(A_{\Rbb}(\Dbb))$ by Theorem~\ref{t_homot}.
\end{proof}

\subsubsection{Ball algebra $A(B^m)$, polydisk algebra $A(\Dbb^m)$, $m\ge 2$,
and infinite polydisk algebra $A(\Dbb^\infty)$}
Let $B^m$ denote the unit ball of $\mathbb{C}^m$, $m\ge 2$.
The ball algebra $A(B^m)$ and the polydisk algebra $A(\Dbb^m)$
are defined analogously to the disk-algebra $A(\Dbb)$.
By \cite[Corollary~3.13]{CS87},
\[\bsr(A(B^m)) = \bsr(A(\Dbb^m)) = \left[\frac{m}{2} \right] +1,\quad  m\ge 2.
\]
The infinite polydisk algebra $A(\Dbb^\infty)$
is the uniform closure of the algebra generated by the coordinate functions $z_1, z_2, \dots$
on the countably infinite closed polydisk $\overline{\Dbb}^\infty = \overline{\Dbb}\times \overline{\Dbb}\dots$.
Proposition~1 from \cite{Mo92} guarantees that $\bsr(A(\Dbb^\infty))=\infty$.
Large or infinite Bass stable rank of the algebras under consideration
is compatible with the following result.

\begin{proposition}
Let $n=2,3,\dots$. Then
\begin{align*}
  \EE_n(A(B^m)) &= \sln(A(B^m)), \quad m=2,3,\dots, \infty, \\
  \EE_n(A(\Dbb^m)) &= \sln(A(\Dbb^m)), \quad m=2,3,\dots, \infty.
\end{align*}
\end{proposition}
\begin{proof}
It suffices to repeat the argument used in the proof of Corollary~\ref{c_AD} or Proposition~\ref{p_A_Rbb}.
\end{proof}

\subsubsection{Algebra $H^\infty_{\Rbb}(\Dbb)$}
It is proved in \cite{MW09} that $\bsr(H^\infty_{\Rbb}(\Dbb))=2$. We have not been able to determine the connected component of the identity in $\sln(H^\infty_{\Rbb}(\Dbb))$.

\begin{prb} Is any element in $\sln(H^\infty_{\Rbb}(\Dbb))$ null-homotopic?

\end{prb}

\section{Invertible matrices as products of exponentials}\label{s_exp}
Let $R$ be a commutative unital ring.
In the present section, we address the following problem:
whether a matrix $F\in \gln(R)$ is representable as a product of exponentials,
that is, $F = \exp G_1\dots \exp G_k$ with $G_j \in M_n(R)$.
For $n=2$ and matrices with entries in a Banach algebra,
this problem was recently studied in \cite{MR18}.

\subsection{Basic results}
There is a direct relation between the problem under consideration and factorization
of matrices in $\gln(R)$.

\begin{lemma}
Let $X\in \sln(R)$ be a unipotent upper or lower triangular matrix.
Then $X$ is an exponential.
\end{lemma}
\begin{proof}
For $n=2$,
we have
\[
\exp\begin{pmatrix}
  0      &a \\
  0      &0
\end{pmatrix} =
\begin{pmatrix}
  1      &a \\
  0      &1
\end{pmatrix}.
\]

Let $n\ge 3$. Given $\alpha_1, \alpha_2, \dots$; $\beta_1, \beta_2, \dots$; $\gamma_1, \gamma_2, \dots$,
we will find $a_1, a_2, \dots$; $b_1, b_2, \dots$; $c_1, c_2, \dots$ such that
\[
\begin{pmatrix}
   1       &\alpha_1   &\alpha_2 &\alpha_3 &\dots \\
           &1          &\beta_1  &\beta_2 &\ddots\\
           &           &1        &\gamma_1 &\ddots \\
           &\mathbf{0} &  &1 &\ddots\\
           &     &     &          &\ddots
\end{pmatrix}
=
\exp
\begin{pmatrix}
   0       &a_1   &a_2 &a_3 &\dots \\
           &0          &b_1  &b_2 &\ddots\\
           &           &0        &c_1 &\ddots \\
           &\mathbf{0} &  &0 &\ddots\\
           &     &     &          &\ddots
\end{pmatrix}.
\]
Put $a_1=\alpha_1$, $b_1=\beta_1$, \dots.
Next, we have $a_2 = \alpha_2 - f(a_1, b_1)
= \alpha_2 - f(\alpha_1, \beta_1)$.
Analogously, we find $b_2$, $c_2$, \dots.
To find $a_3$, observe that $a_3 = \alpha_3 - f(a_1, a_2, b_1, c_2)$.
Since $f$ depends on $a_i$, $b_i$, $c_i$ with $i<3$,
we obtain $a_3 = \alpha_3 - \widetilde{f}(\alpha_1, \alpha_2, \beta_1, \beta_2, \gamma_1, \gamma_2)$,
and the procedure continues.
So, the equation under consideration is solvable for any
$\alpha_1, \alpha_2, \dots$; $\beta_1, \beta_2, \dots$.
\end{proof}

\begin{corollary}\label{c_exp1}
Assume that $\sln(R) = \EE_n(R)$
and every element in $\EE_n(R)$ is a product of $N(R)$
unipotent upper or lower triangular matrices. Then every element in $\sln(R)$
is a product of $N(R)$ exponentials.
\end{corollary}

\begin{corollary}\label{c_exp2}
Let the assumptions of Corollary~\ref{c_exp1} hold.
Suppose in addition that every invertible element in $R$
admits a logarithm.
Then every $X\in \gln(R)$ is a product of $N(R)$ exponentials.
\end{corollary}
\begin{proof}
Let $X\in \gln(R)$.
So, $\det X \in R^*$ and $\ln\det X$ is defined.
Therefore, $\det X = f^n$ for appropriate $f\in R^*$ and
\[
\begin{pmatrix}
  f^{-1}     &       &\mathbf{0} \\
             &\ddots &            \\
  \mathbf{0} &       &f^{-1}
\end{pmatrix}
X \in \sln(R).
\]
Applying Corollary~\ref{c_exp1}, we obtain
\begin{align*}
  X
&=
  \begin{pmatrix}
  f     &       &\mathbf{0} \\
             &\ddots &            \\
  \mathbf{0} &       &f
  \end{pmatrix} \exp Y_1\dots \exp Y_N\\
&= \exp\left[ \begin{pmatrix}
  \ln f     &       &\mathbf{0} \\
             &\ddots &            \\
  \mathbf{0} &       &\ln f
  \end{pmatrix}
   +Y_1 \right]\exp Y_2\dots \exp Y_N,
\end{align*}
as required.
\end{proof}

\subsection{Rings of holomorphic functions on Stein spaces}

\begin{corollary}\label{c_stein}
Let $\Omega$ be a Stein space of dimension $k$ and let $X\in \gln(\OO(\Omega))$.
Then there exists a number $E(k, n)$ such that the following properties are equivalent:
\begin{itemize}
  \item[(i)] $X$ is null-homotopic;
  \item[(ii)] $X$ is a product of $E(k, n)$ exponentials.
\end{itemize}
\end{corollary}
\begin{proof}
By \cite[Theorem~2.3]{IK12}, any null-homotopic $F\in \sln(\OO(\Omega))$
is a product of $N(k, n)$ unipotent upper or lower triangular
matrices.
So, arguing as in the proof of Corollary~\ref{c_exp2},
we conclude that (i) implies (ii) with $E(k,n) \le N(k,n)$
The reverse implication is straightforward.
\end{proof}
The numbers $N(k, n)$ are not known in general.
If the dimension $k$ of the Stein space is fixed, then
the dependence of $N(k,n)$ on the size $n$ of the matrix is easier to handle.
Certain $K$-theory arguments guarantee that the number of unipotent matrices needed for factorizing an element in $\sln(\OO(\Omega))$ is a non-increasing function of $n$ (see \cite{DV88}). So, as done in \cite{Brud},
combining the above property and results from \cite{IK12PAMS}, we obtain the following estimates:
\begin{align*}
 E(1,n)\le N(1,n) & =  4 \textrm{\ for all\ } n, \\
 E(2,n)\le  N(2,n) &\le  5 \textrm{\ for all\ } n, \textrm{\ and\ }\\
\textrm{for each\ } k, \textrm{\ there exists\ } n(k) \textrm{\ such that\ }
   E(k,n)\le N(k,n) &\le 6 \textrm{\ for all\ } n \ge n(k).
\end{align*}
In Section~\ref{ss_3exp}, we in fact improve on that: we show $E(1,2) \le 3$.
In general, it seems that the number of exponentials $E(k,n)$ to factorize an element in $\gln(\OO(\Omega))$
is less than the number $N(k,n)$ needed to write
an element in  $\sln(\OO(\Omega))$ as a product of unipotent upper or lower triangular
matrices.

Also, remark that (ii) implies (i) in Corollary~\ref{c_stein}
for any algebra $R$ in the place of the ring of holomorphic functions.
Assume that the algebra  $R$ has a topology. Then a topology on $\gln(R)$ is naturally induced and
the implication (i)$\Rightarrow$(ii) means that any product of exponentials is contained in the connected component
of the identity (also known as the principal component) of $\gln(R)$.
The reverse implication is a difficult question, even without a uniform bound on the number of exponentials.

\subsection{Rings $R$ with $\bsr(R)=1$}
Combining Theorem~\ref{t_sbr1_2} and Corollary~\ref{c_exp2},
we recover a more general version of Theorem~7.1(3) from \cite{MR18},
where $R$ is assumed to be a Banach algebra.
Moreover, we obtain similar results for larger size matrices.

\begin{corollary}\label{c_glz_4}
Let $R$ be a commutative unital ring, $\bsr R =1$,
and let every $x\in R^*$ admit a logarithm.
Then every element in $\glz(R)$ is a product of $4$ exponentials.
\end{corollary}

\begin{corollary}
Let $R$ be a commutative unital ring, $\bsr R =1$,
and let every $x\in R^*$ admit a logarithm.
Then every element in $\gln(R)$, $n\ge 3$, is a product of $6$ exponentials.
\end{corollary}

\begin{proof}
For $n=3$, it suffices to combine Theorem~\ref{t_sbr1_2} and Corollary~\ref{c_exp2}.

Now, assume that $n\ge 4.$
Let $\ut_m$ denote the number of unipotent matrices needed to factorize
any element in $\sla_m(R)$ starting with an upper triangular matrix.
Theorem~20(b) in \cite{DV88} says that any element in $\sln(R)$ is a product of $6$ exponentials for
\[
n \ge \min \left(m \left[ \frac{\ut_m (R) +1} {2} \right]\right),
\]
 where the minimum is taken over all
$m \ge \bsr R +1$. In our case the minimum is taken over $m\ge 2$ and the number $\ut_2 (R) =4 $ by the proof of
Theorem \ref{t_sbr1_2}.
Since $n\ge 4$, the proof is finished.
\end{proof}

Corollary~\ref{c_glz_4} applies to the disk algebra and also to
the rings $\OO(\Cbb)$ and $\OO(\Dbb)$ of holomorphic functions.
Indeed, the identity $\bsr(\OO(\Omega))=1$ for an open Riemann
surface follows from the strengthening of the classical Wedderburn lemma
(see \cite[Chapter 6, Section 3]{R}; see also \cite{IK12} or \cite{Br18}).
However, for $R=\OO(\Cbb)$ and $R=\OO(\Dbb)$,
the number $4$ is not optimal; see Section~\ref{ss_3exp} below.
Also, it is known that the optimal number is at least $2$
(see \cite{MR18}). So, we arrive at the following natural question:

\begin{prb}
Is any element of $\glz(\OO(\Dbb))$ or $\glz(\OO(\Cbb))$
a product of two exponentials?
\end{prb}

\subsection{Products of $3$ exponentials}\label{ss_3exp}
In this section, we prove the following result.

\begin{proposition}\label{p_riemann}
Let $\Omega$ be an open Riemann surface.
Then every element in $\slz(\OO(\Omega))$
is a product of $3$ exponentials.
\end{proposition}

We will need several auxiliary results. The first theorem is a classical one \cite{Flor}.

\begin{theorem}[Mittag-Leffler Interpolation Theorem]\label{t_MLeff}
Let $\Omega$ be an open Riemann surface
and let $\{z_i\}_{i=1}^\infty$ be a discrete closed subset of $\Omega$.
Assume that a finite jet
\begin{equation}\label{e_jet}
J_i(z) = \sum_{j=1}^{N_i} b_j^{(i)} (z- z_i)^j
\end{equation}
is defined in some local coordinates for every point $z_i$.
Then there exists $f\in \OO(\Omega)$ such that
\begin{equation}\label{e_jet_app}
f(z) - J_i(z) = o (|z- z_i|^{N_i})\quad \textrm{as\ } z\to z_i,\ i=1,2\dots.
\end{equation}
\end{theorem}

\begin{corollary}\label{c_MLeff}
Under assumptions of Theorem~\ref{t_MLeff}, suppose that $b_0^{(i)} \neq 0$ in \eqref{e_jet}
for $i=1,2,\dots$.
Then there exist $f, g\in \OO(\Omega)$ such that \eqref{e_jet_app} holds
and $f = e^g$.
\end{corollary}
\begin{proof}
Let $b_0= b_0^{(i)}$ for some $i$.
Since $b_0\neq 0$, there exists a logarithm $\ln$
in a neighborhood of $b_0$.
So, $\ln$ is a local biholomorphism which induces
a bijection between jets of $f$ and $g:= \ln f$.
\end{proof}

In ``modern'' language, the  proof of Corollary~\ref{c_MLeff}  uses the fact that $\Cbb^*$ is an Oka manifold
(we refer the interested reader to \cite{For}).
Thus
for any Stein manifold $X$ and an analytic subset $Y\subset X$,
a (jet of) holomorphic map $f: Y \to \Cbb^*$ (along $Y$)
extends to a holomorphic map $f: X \to \Cbb^*$
if and only if it extends continuously.
The obstruction for a continuous extension is an element of the relative homology group $H_2(X, Y, \mathbb{Z})$.
Observe that,
for any discrete subset $Y$ of a $1$-dimensional Stein manifold $X$,
we have
$H_2(X, Y, \mathbb{Z})=0$
because $H_2(X, \mathbb{Z}) = H_1 (Y, \mathbb{Z})=0$.
This is the point where the proof of Proposition~\ref{p_riemann} below breaks down when we replace the Riemann surface
$\Omega$ by a Stein manifold of higher dimension. Even a nowhere vanishing continuous function
$\alpha$, as in the proof, does not exist in general.

\begin{lemma}\label{l_double_ev}
Let $\Omega$ be an open Riemann surface and $X\in \glz(\OO(\Omega))$.
Assume that $\lambda\in \OO^*(\Omega)$ is the double eigenvalue of $X$
and $\det X$ has a logarithm in $\OO(\Omega)$.
Then $X$ is an exponential.
\end{lemma}
\begin{proof}
We consider two cases.
\paragraph{Case 1:}
$X(z)$ is a diagonal matrix for all $z\in\Omega$.

We have
\[
X(z)=
 \begin{pmatrix}
  \lambda(z) &0\\
  0         &\lambda(z)
  \end{pmatrix}
= \exp
\begin{pmatrix}
  \alpha(z) &0\\
  0         &\alpha(z)
  \end{pmatrix}.
\]

\paragraph{Case 2:}
$X(z)$ is not identically diagonal.

Either the first or the second line in $X(z) - \lambda(z)I$,
say $(h(z), g(z))$, is not identical zero. So,
\[
v_1(z)=
 \begin{pmatrix}
  -g(z) \\
  h(z)
  \end{pmatrix}
\]
is a holomorphic eigenvector for $X(z)$ except those
points $z\in\Omega$ for which $v_1(z) = \mathbf{0}$.
Construct a function $f(z)\in \OO(\Omega)$ such that
its vanishing divisor is exactly $\min(\textrm{ord\,} g, \textrm{ord\,} h)$.
Then
\[
v(z) = \frac{1}{f(z)} v_1 (z)
\]
is a holomorphic eigenvector for $X(z)$, $z\in\Omega$.

Now, choose a matrix $P(z) \in \glz(\OO(\Omega))$
with first column $v(z)$.
Then the matrix $P^{-1}(z) X(z) P(z)$
has the following form:
\[
 \begin{pmatrix}
  \lambda(z) &\beta(z)\\
  0          &\lambda(z)
  \end{pmatrix}
  =
  \exp
\begin{pmatrix}
  \frac{1}{2}\gamma(z) &\frac{\beta(z)}{\lambda(z)}\\
  0                    &\frac{1}{2}\gamma(z)
  \end{pmatrix}
\]
Thus,
\[
X(z)= \exp
P(z)
\begin{pmatrix}
  \frac{1}{2}\gamma(z) &\frac{\beta(z)}{\lambda(z)}\\
  0                    &\frac{1}{2}\gamma(z)
  \end{pmatrix}
  P^{-1}(z),
\]
as required.
\end{proof}

\begin{proof}[Proof of Proposition~\ref{p_riemann}]
Let
\[
X=
\begin{pmatrix}
  a      &b \\
  c      &d
\end{pmatrix}
\in \slz(R),
\]
that is, $ad - bc =1$.
We are looking for $\alpha\in R^*$ and $\beta\in R$ such that
the matrix
\[
X
\begin{pmatrix}
  \alpha^2      &\beta \\
  0             &1
\end{pmatrix}
=
\begin{pmatrix}
  \alpha^2 a      &\beta a + b \\
  \alpha^2 c      &\beta c + d
\end{pmatrix}
:=Y
\]
has a double eigenvalue.

\paragraph{Case 1:} $c=0$.
We have
\[
X=
\begin{pmatrix}
  a      &b \\
  0      &a^{-1}
\end{pmatrix}.
\]
It suffice to observe that
\[
\begin{pmatrix}
  a      &b \\
  0      &a^{-1}
\end{pmatrix}
\begin{pmatrix}
  a^{-2} &0 \\
  0      &1
\end{pmatrix}
=
\begin{pmatrix}
  a^{-1}      &b \\
  0           &a^{-1}
\end{pmatrix}
\]
has the double eigenvalue $a^{-1}$.

\paragraph{Case 2:} $c\neq 0$.
The matrix $Y$ has a double eigenvalue if $4 \det Y = (\mathrm{tr\,} Y)^2$,
that is,
\begin{equation}\label{e_double_ev}
(\alpha^2 a + \beta c + d)^2 = 4\alpha^2.
\end{equation}
Put
\[
\beta= \frac{2\alpha - a\alpha^2 - d}{c}.
\]
Clearly, $\beta$ is a formal solution of \eqref{e_double_ev}.
Below we show how to construct
$\alpha(z) =\exp(\widetilde{\alpha}(z))\in \OO^*(\Omega)$
such that $\beta$ is holomorphic.

Let $\{z_i\}\subset \Omega$ be the zero set of $c(z)$.
Fix $i$ and $z_i\in \Omega$.  Let $c(z_i) = \dots = c^{(k)}(z_i)=0$, and $c^{(k+1)}(z_i)\neq 0$.
Observe that $a(z_i)\neq 0$. So, define $\alpha(z)$, in a neighborhood of $z_i$, as
$1/a(z)$ up to a sufficiently high order, namely,
\begin{equation}\label{e_a_al}
a(z)\alpha(z) = 1 + (z-z_i)^k h(z),
\end{equation}
where $h(z)$ is holomorphic in a neighborhood of $z_i$.
Since $ad -bc =1$, we have $1-ad = (z-z_i)^k g(z)$. Therefore,
\begin{align*}
  2a\alpha - a^2 \alpha^2 -ad
&= -(1-a\alpha^2)^2 + 1- ad \\
&=- (z-z_0)^{2k} h^2(z) + (z-z_0)^k g(z)
\end{align*}
vanishes of order $k$ at $z_i$.
Hence, $2\alpha - a\alpha^2 - d$ also vanishes of order $k$ at $z_i$.

So, we have constructed $\alpha(z)$ locally as finite jets
$J_i(z)$ defined by \eqref{e_jet} with $b_0^{(i)}\neq 0$
in some local coordinates for every point $z_i$, $i=1,2,\dots$.
Now, Corollary~\ref{c_MLeff} provides $\widetilde{\alpha}\in \OO(\Omega)$ such that
$\alpha(z) =\exp(\widetilde{\alpha}(z))\in \OO^*(\Omega)$ and \eqref{e_a_al} holds.
Hence, $\beta$ is holomorphic.

So, the matrix
\[
X
\begin{pmatrix}
  \alpha^2      &\beta \\
  0             &1
\end{pmatrix}
:=Y
\]
has a double eigenvalue and $\det Y$ admits a logarithm.
Thus, applying Lemma~\ref{l_double_ev}, we conclude that $Y$ is an exponential.
To finish the proof of the proposition, it remains observe that
\[
\begin{pmatrix}
  \alpha^2      &\beta \\
  0             &1
\end{pmatrix}
=
\begin{pmatrix}
  \alpha      &0 \\
  0           &\alpha^{-1}
\end{pmatrix}
\begin{pmatrix}
  \alpha      &\beta\alpha^{-1} \\
  0            &\alpha
\end{pmatrix},
\]
where both multipliers on the right hand side are exponentials.
\end{proof}

\begin{corollary}
Let $X\in \glz(\OO(\Omega))$.
The following properties are equivalent:
\begin{itemize}
  \item[(i)] $X$ is a product of $3$ exponentials;
  \item[(ii)] $\det X$ is an exponential;
  \item[(iii)] $X$ is null-homotopic.
\end{itemize}
\end{corollary}
\begin{proof}
Clearly, (i)$\Rightarrow$(iii).
Now, assume that $X$ is null-homotopic.
Then $\det X$ is homotopic to the function $f\equiv 1$.
Since $\exp: \Cbb \to \Cbb^*$ is a covering,
we conclude that $\det X(z) =\exp(h(z))$ with $h\in\OO(\Omega)$.
So, (iii) implies (ii). The implication (ii)$\Rightarrow$(i)
is standard; see, for example, the proof of Corollary~\ref{c_exp2}.
\end{proof}

\end{document}